\documentclass[small]{article}

\usepackage{amsmath,amstext,amsthm,amsfonts}
\usepackage{makeidx}

\usepackage{marginnote}
\usepackage{bbding}

\usepackage{wasysym}
\usepackage{amssymb,latexsym}
\usepackage{txfonts, pxfonts}
\usepackage[english]{babel}
\usepackage[pdftex]{graphicx}
\usepackage{xcolor}
\numberwithin{equation}{section}
\newtheorem{theorem}{\textbf{Theorem}}[section]
\newtheorem{proposition}[theorem]{\textbf{Proposition}}

\newtheorem{lemma}[theorem]{\textbf{Lemma}}
\newtheorem{corollary}[theorem]{Corollary}
\newtheorem{conjecture}[theorem]{\textbf{Conjecture}}
\theoremstyle{definition}
\newtheorem{definition}[theorem]{\textbf{Definition}}
\theoremstyle{remark}
\newtheorem{remark}[theorem]{\it{Remark}}
\newenvironment{notation}[1][Notation.]{\begin{trivlist}
\item[\hskip \labelsep {\bfseries #1}]}{\end{trivlist}}

\def\e{\epsilon}
\def\ei{\epsilon_i}
\def\R{\mathbb{R}}
\def\Rn{{\mathbb{R}}^n_+}

\def\d{\partial}

\def\fermi{\psi_i:B^+_{\delta}(0)\to M}
\def\fermilinha{\psi_i:B^+_{\delta'}(0)\to M}
\def\a{\alpha}
\def\b{\beta}

\def\Sei{S^+_{\delta\ei^{-1}}}
\def\Bei{B^+_{\delta\ei^{-1}}}
\def\Dei{D_{\delta\ei^{-1}}}
\def\Beilinha{B^+_{\delta'\ei^{-1}}}
\def\Deilinha{D_{\delta'\ei^{-1}}}

\def\ba{\begin{align}}
\def\ea{\end{align}}
\def\bp{\begin{proof}}
\def\ep{\end{proof}}

\def\cmedia{h}
\def\ds{d\sigma}

\renewcommand{\(}{\left(}
\renewcommand{\)}{\right)}

\begin{document}

\title{A priori estimates for negative constant scalar curvature conformal metrics with positive constant boundary mean curvature}
\date{}

\author{\textsc{S\'ergio Almaraz}\footnote{Partially supported by grant 201.049/2022, FAPERJ/Brazil.}  \textsc{and Shaodong Wang}\footnote{Partially supported by NSFC 12001364 and the Fundamental Research Funds for the Central Universities, No.30924010839.}}

\maketitle

\begin{abstract}
On a compact Riemannian manifold with boundary, we study the set of conformal metrics of negative constant scalar curvature in the interior and positive constant mean curvature on the boundary. Working in the case of positive Yamabe conformal invariant, we prove that this set is a priori bounded in the three-dimensional case and in the locally conformally flat with umbilical boundary case in any dimension not less than three.

\end{abstract}




\section{Introduction}\label{sec:intr}

Let $(M,g)$ be an $n$-dimensional Riemannian manifold with boundary $\d M$, $n\geq 3$. Denote by $R_g$ its scalar curvature and by $\Delta_g$ its Laplace-Beltrami operator which is the Hessian trace. By $h_g$ we denote the boundary mean curvature with respect to the inward normal vector $\eta$, i.e. $h_g=-\frac{1}{n-1}\rm{div}_g\eta$.
In this paper, we study the set of positive solutions to the equations
\begin{align}\label{main:equation:0}
\begin{cases}
L_{g}u+Ku^{\frac{n+2}{n-2}}=0,&\text{in}\:M,
\\
B_{g}u+cu^{\frac{n}{n-2}}=0,&\text{on}\:\partial M,
\end{cases}
\end{align}
where $K, c\in \mathbb R$. Here, $L_g=\Delta_g-(n-2)R_g/(4(n-1))$ is the conformal Laplacian and $B_g=\d/\d\eta-(n-2)h_g/2$ is the conformal boundary operator.
It is well known that a smooth solution $u>0$ of the problem \eqref{main:equation:0} represents a conformal metric $\tilde g=u^{\frac{4}{n-2}}g$ with scalar curvature $R_{\tilde g}=4(n-1)K/(n-2)$ and boundary mean curvature $h_{\tilde g}=2c/(n-2)$.

The problem of finding a solution to equations \eqref{main:equation:0} is known as the Escobar-Yamabe problem and was originally proposed in \cite{escobar2, escobar3, escobar4}. Canonical solutions of \eqref{main:equation:0} are obtained for $(M,g)=(\mathbb B^n, \delta_{\R^n})$, the unit ball in $\mathbb R^n$ with the Euclidean metric. In this case, Escobar proved in \cite{escobar1} that, up to a multiplicative constant, the only solutions of those equations are so that $\tilde g=u^{\frac{4}{n-2}}g$ is isometric to a spherical cap of the unit sphere $\mathbb S^n$, to a ball in the Euclidean space, or to a geodesic ball in the hyperbolic space $\mathbb H^n$, according  to $K>0$, $K=0$, or $K<0$, respectively. Moreover, in the latter case one necessarily has $c>\sqrt{-(n-2)K/n}$.

In our previous works \cite{almaraz-queiroz-wang} and \cite{almaraz-wang}, we studied compactness of the set of positive solutions to equations \eqref{main:equation:0} in the cases $K=0$ and $K>0$, respectively. In this paper, we are interested in the set of solutions to equations \eqref{main:equation:0} when $K$ is negative. Our main result, which is Theorem \ref{compactness:thm} below, proves that this set is a priori bounded in the $C^{2,\alpha}(M)$ topology under certain conditions.  

In the case of manifolds without boundary, the question of compactness in the Yamabe problem was raised by Schoen in a topics course at Stanford University in 1988. This question has attracted the interest of important mathematicians and a final answer was given in the works \cite{brendle2, brendle-marques, khuri-marques-schoen}. While in the case without boundary the negativity of the scalar curvature implies uniqueness of the solution to the Yamabe equation via the maximum principle, the same argument cannot be applied in the case of non-empty boundary when its mean curvature is positive.

We will adopt the normalization $c=n-2$ and assume $-n(n-2)<K< 0$. Throughout our paper, we write $K=-n(n-2)\kappa$ where $0< \kappa<1$. In other words, we will study the system
\begin{align}\label{main:equation:1}
\begin{cases}
L_{g}u-n(n-2)\kappa u^{\frac{n+2}{n-2}}=0,&\text{in}\:M,
\\
B_{g}u+(n-2)u^{\frac{n}{n-2}}=0,&\text{on}\:\partial M.
\end{cases}
\end{align}

Let $[g]$ be the conformal class of $g$. We will be specializing to the case of positive conformal invariant $Q_g(M)$ which is defined by Escobar in \cite{escobar3} as
$$
Q_g(M)=\inf_{\tilde g\in [g]}\frac{\int_M R_{\tilde g}dv_{\tilde g}+2(n-1)\int_{\partial M} h_{\tilde g}d\sigma_{\tilde g}}{\big(\int_Mdv_{\tilde g}\big)^{\frac{n-2}{n}}}.
$$
Here, $dv$ and $d\sigma$ denote the volume and area forms, respectively, and $Q_g(M)$ is proven to be finite. We say that $(M,g)$ is of {\it{positive}}, {\it{zero}} or {\it{negative type}}, according to the sign of $Q_g(M)$. It is important to mention that all the canonical examples mentioned above are of positive type since they are all conformally diffeomorphic to the round hemisphere $\mathbb S^n_+$ which is obviously of positive type.

Our main result is the following:

\begin{theorem}\label{compactness:thm}
	Let $(M,g)$ be a compact $n$-dimensional Riemannian manifold with boundary $\d M$.  Suppose that $M$ is of positive type and it is not conformally equivalent to $\mathbb S^n_+$. 
Assume further that $n=3$ or $M$ is locally conformally flat with $\partial M$ umbilic and $n\geq 3$.
For any $0<\kappa<1$, there exists $C(M,g,\kappa)>0$ such that for any solution $u>0$ of (\ref{main:equation:1}) we have
	$$C^{-1}\leq u\leq C\:\:\:\:\: \text{and}\:\:\:\:\:\|u\|_{C^{2,\a}(M)}\leq C\,,$$
	for some $0<\a<1$.
\end{theorem}

The hypothesis of $M$ not being conformally equivalent to $\mathbb S^n_+$ is necessary as the canonical examples mentioned above constitute in fact a family of blowing-up solutions (see Section \ref{sec:pre}).
The hypothesis $0< \kappa<1$ is justified by the canonical example given by the geodesic ball in the hyperbolic space which provides the pattern for blowing up sequences of solutions (see Section \ref{sec:pre}). On the other hand, in the case $\kappa>1$, the lack of solutions to the corresponding Euclidean equations in $\mathbb R^n_+$ allows one to directly rule out the blow-up behaviour. This was done by Han and Li in \cite[Theorem 0.3]{han-li1} for the case of manifolds of negative type. 
As for the case $\kappa=1$,  although there exists a family of solutions to the Euclidean Yamabe equations in $\mathbb R^n_+$, they do not represent compact Riemannian manifolds as they stand for domains in the hyperbolic space bounded by horospheres (see Remark \ref{rmk:horo} below).

Let us briefly discuss Theorem \ref{compactness:thm} in view of previous compactness results for equations \eqref{main:equation:0}.
In the case $K>0$, a compactness theorem was established in \cite{almaraz-wang} for three-dimensional manifolds and in \cite{han-li1} for locally conformally flat manifolds with umbilical boundary. Assuming also $c=0$ and the boundary umbilic, a compactness theorem was proved in \cite{disconzi-khuri} for dimensions $3\leq n\leq 24$.
The case $K=0$ was studied in \cite{almaraz-queiroz-wang, ahmedou-felli2} in dimension three, in \cite{kim-musso-wei} for dimensions four and five, and for locally conformally flat manifolds with umbilical boundary in \cite{ahmedou-felli1}. Theorem \ref{compactness:thm} can be regarded as an extension of the results in \cite{almaraz-queiroz-wang, almaraz-wang, ahmedou-felli1, han-li1} to negative values of $K$.
The novelty here is the use, for compactness problems, of a geodesic ball in the hyperbolic space as the model for blow-up sequences of solutions by means of the Liouville theorem in \cite{chipot-fila-shafrir} (see \cite{cruz-malchiodi-ruiz, cruz-pistoia-vaira} for the use of the same model for similar problems).

Different from the cases $K=0$ and $K>0$, no a priori estimates are obtained here when $K<0$ for subcritical exponents replacing $(n+2)/(n-2)$ and $n/(n-2)$ in equations \eqref{main:equation:0}. This is due to the lack of Liouville type results for this situation, to the best of our knowledge.

The idea of the proof of Theorem \ref{compactness:thm} follows the approach originally proposed by Schoen when he conjectured the compactness of the full set of solutions to the classical Yamabe problem for manifolds without boundary. It consists in ruling out the blow-up behaviour by means of a Pohozaev identity. 
This method was successfully implemented by Khuri-Marques-Schoen \cite{khuri-marques-schoen} to prove the compactness conjecture of Yamabe problems on manifolds without boundary in dimension $3\leq n\leq 24$, while counterexamples in dimensions $n\geq 25$ were provided by Brendle \cite{brendle2} and Brendle-Marques \cite{brendle-marques} (we refer the reader to \cite{berti-malchiodi} for non-smooth counterexamples).
As discussed in the introduction section of \cite{almaraz-wang}, there are some technical difficulties in implementing this approach for manifolds with boundary. The main one is the lack of control, in general, of the correction term obtained in Proposition \ref{Linearized} below, and of the Green's function expansion for the conformal Laplacian. The hypotheses of either low dimension or the manifold being locally conformally flat with umbilical boundary  are used to overcome those difficulties.

As in the case of manifolds without boundary, we also expect the existence of a critical dimension where the compactness for the equations \eqref{main:equation:0} holds in general up to that dimension and fails above it. Blowing up examples in high dimensions were obtained in \cite{almaraz5} for the case $K=0$ and in \cite{chen-wu, disconzi-khuri} for the case $K>0$. Those examples are based on the constructions in \cite{brendle2, brendle-marques} for the case of manifolds without boundary. We believe this constructions can be also extended to our setting with $K<0$.

We now turn to existence results for equations \eqref{main:equation:0}.
For $K=0$, a minimization argument based on an Aubin type inequality was used in \cite{almaraz1, chen, escobar2, coda1, coda2}. 
In the case $K>0$, there are several results using minimization arguments: see \cite{araujo, brendle-chen, chen-sun, escobar2, escobar4} (we refer to \cite{mayer-ndiaye1, mayer-ndiaye2} for the use of topological methods).
However, this minimization approach does not allow one to prescribe both constants $K>0$ and $c\in\mathbb R$ at the same time in equations \eqref{main:equation:0} (unless $c=0$). 
This problem is known as the Han-Li conjecture and was proposed in \cite{han-li1}, where the existence of solutions follows from the compactness result by means of topological degree theory in the case of locally conformally case with umbilical boundary.  
A direct approach to that conjecture is the use of a mountain pass lemma in \cite{chen-ruan-sun, han-li2}.
Unlike the cases $K\geq 0$, there are few existence results to \eqref{main:equation:0} when $K<0$. 
Assuming the manifold is of negative type, an existence theorem was proved in \cite{han-li1} using degree theory. 
Recently, in \cite{cruz-malchiodi-ruiz}, the existence of solutions to \eqref{main:equation:0} on manifolds of non-positive types was obtained via variational methods. 
To our best knowledge, there is no existence results in the case $K<0$ on manifolds of positive type.
In view of the above model based on geodesic balls in $\mathbb H^n$, this motivates us to propose the following:
\begin{conjecture}\label{existence:conj}
Let $(M,g)$ be a compact Riemannian manifold with boundary $\d M$, of positive type, and dimension $n\geq 3$. Given $K<0$ and $c>\sqrt{-(n-2)K/n}$, there exists a smooth solution $u>0$ to equations \eqref{main:equation:0}.
\end{conjecture}
Solutions to \eqref{main:equation:0} correspond to critical points of the following functional which is defined for any non-negative $u\in H^1(M)$:
$$
I(u)=E(u)-\frac{n-2}{n}K\int_{M}u_+^{\frac{2n}{n-2}}-\frac{n-2}{n-1}c\int_{\d M}u_+^{\frac{2(n-1)}{n-2}}
$$
where $E(u)=\int_{M}\(|\nabla u|^2+\frac{n-2}{4(n-1)}R_gu^2\)+\frac{n-2}{2}\int_{\d M}h_gu^2$ and we are omitting the volume and area forms for simplicity.
Observe that $I(0)=0$ and $I(u)>0$ for small $\|u\|_{H^1(M)}>0$. If $K>0$ and $u_+\nequiv 0$, then $\lim_{t\to\infty}I(tu)=-\infty$. When below the energy level of the spherical caps, this allows one to apply the mountain pass lemma to find critical points of $I$ as done in  \cite{chen-ruan-sun, han-li2}. 
However, in the case $K<0$, $\lim_{t\to\infty}I(tu)=-\infty$ does not hold which prevents us from using the same methods as the positive case.
We believe Conjecture \ref{existence:conj} is an interesting problem that deserves further investigation.

One potential application of the a priori estimates obtained in Theorem~\ref{compactness:thm} is to prove Conjecture  \ref{existence:conj} using the topological degree theory as done in \cite{han-li1}. The idea is to deform equations \eqref{main:equation:0} to another system of equations where the degree is known to be non-zero,  while preserving the a priori bound.
We propose a connection with the case $\kappa=0$ where the degree is $-1$ as calculated in \cite{ahmedou-felli1}:
\begin{conjecture}\label{compactness:conj}
	Let $(M,g)$ be a compact $n$-dimensional Riemannian manifold with boundary $\d M$.  Suppose that $M$ is of positive type and it is not conformally equivalent to $\mathbb S^n_+$. 
Assume further that $n=3$ or $M$ is locally conformally flat with $\partial M$ umbilic and $n\geq 3$.
Given small $\delta>0$, there exists $C(M,g,\delta)>0$ such that for any solution $u>0$ of (\ref{main:equation:1}) with $0\leq \kappa\leq 1-\delta$  we have
	$$C^{-1}\leq u\leq C\:\:\:\:\: \text{and}\:\:\:\:\:\|u\|_{C^{2,\a}(M)}\leq C\,,$$
	for some $0<\a<1$.
\end{conjecture}
These a priori estimates hold for $0<\kappa<1$, according to Theorem \ref{compactness:thm}, and for $\kappa=0$, according to \cite{almaraz-queiroz-wang}. This supports our conjecture. 

The paper is organized as follows. In Section \ref{sec:pre}, we provide some preliminary results about the model solutions in the Euclidean half space. The definitions and basic properties of isolated and isolated simple blow-up points are collected in Section \ref{sec:isolated}. In Section \ref{sec:blowup:estim}, we prove a refined approximation result for blow-up sequences in terms of the standard Euclidean solutions plus some correction term. The sign restriction for an integral appearing in a Pohozaev type identity is then obtained using such refined estimates. This is done in Section \ref{sec:sign:restr}. Finally in Section \ref{sec:pf:thm}, we prove our main result Theorem \ref{compactness:thm}.


\section{Preliminaries}\label{sec:pre}

\begin{notation}
Throughout this work we will make use of the index notation for tensors and adopt the summation convention whenever confusion is not possible. When dealing with coordinates on manifolds with boundary, we will use indices $1\leq j,l\leq n-1$ and $1\leq a,b,c,d\leq n$. The second fundamental form of the boundary will be denoted by $\pi_{jl}=-g(\nabla_{\partial_j}\eta, \partial_l)$ in local coordinates. 
By $\Rn$, we will denote the half-space $\{z=(z_1,...,z_n)\in \R^n;\:z_n\geq 0\}$. If $z\in\Rn$ we set $\bar{z}=(z_1,...,z_{n-1})\in\R^{n-1}\cong \d\Rn$. We define 
$$B^+_{\delta}(0)=\{z\in\Rn\,;\:|z|<\delta \}, \qquad  S^+_{\delta}(0)=\{z\in\Rn\,;\:|z|=\delta \},$$
and $$D_{\delta}(0)=B^+_{\delta}(0)\cap \d\Rn=\{z\in\d\Rn\,;\:|z|<\delta \}.$$
Thus, $\d B^+_{\delta}(0)=D_{\delta}(0)\cup S^+_{\delta}(0)$.
We also denote $B^+_{\delta}=B^+_{\delta}(0)$,  $S^+_{\delta}=S^+_{\delta}(0)$  and $D_{\delta}=D_{\delta}(0)$ for short. 
Finally, by $O(s)$ we mean a term whose absolute value is uniformly bounded by $s$.

\vspace{0.2cm}
\noindent
{\bf{Agreement.}} As we are assuming $Q(M)>0$, without loss of generality we will assume that $R_g>0$ in $M$ and $h_g=0$ along $\partial M$ as one can simply change \eqref{main:equation:1} by a conformal factor (see \cite{escobar3}).

\end{notation}


Next we recall the definition of Fermi coordinates and some expansions for the metric. 

\begin{definition}\label{def:fermi}
Let $x_0\in\d M$ and choose boundary geodesic normal coordinates $(z_1,...,z_{n-1})$, centred at $x_0$, of the point $x\in\d M$.
We say that $z=(z_1,...,z_n)$, for small $z_n\geq 0$, are the {\it{Fermi coordinates}} (centred at $x_0$) of the point $\exp_{x}(z_n\eta(x))\in M$. Here, we denote by $\eta(x)$ the inward unit normal vector to $\d M$ at $x$. In this case, we have a map  $\psi(z)=\exp_{x}(z_n\eta(x))$, defined on a subset of $\Rn$.
\end{definition}

It is easy to see that in these coordinates $g_{nn}\equiv 1$ and $g_{jn}\equiv 0$, for $j=1,...,n-1$. The expansion for $g$ in Fermi coordinates around $x_0$ is given by:
	\begin{align}\label{exp:g}
	g_{jl}(\psi(z))&=\delta_{jl}-2\pi_{jl}(x_0)z_n+O(|z|^2),\notag
	\\
	g^{jl}(\psi(z))&=\delta_{jl}+2\pi_{jl}(x_0)z_n+O(|z|^2),\notag
\\
{\rm{det}}\,g_{ab}\,(\psi(z))&=1-(n-1)h_g(x_0) z_n+O(|z|^2).
	\end{align}


\subsection{The model solutions in $\mathbb R^n_+$}\label{subsec:model}

Our canonical model is given on  $\mathbb R_+^n$ by the function
\begin{equation}\label{eq:Uk}
U_{\kappa}(y)=\big( \sum_{j=1}^{n-1}y_j^2+(y_n+1)^2-\kappa\big)^{\frac{2-n}{2}},\qquad y\in\mathbb R_+^n,
\end{equation}
where $0<\kappa<1$. A direct computation shows that $U_{\kappa}$ solves the system
\begin{align}\label{eq:U}
\begin{cases}
\Delta U_{\kappa}-n(n-2)\kappa U_{\kappa}^{\frac{n+2}{n-2}}=0,&\text{in}\:\mathbb R_+^n,
\\
\displaystyle\frac{\partial U_{\kappa}}{\partial y_n}+(n-2)U_{\kappa}^{\frac{n}{n-2}}=0,&\text{on}\:\mathbb \partial \R_+^n.
\end{cases}
\end{align}
 Since equations \eqref{eq:U} are invariant by translations in the first $n-1$ coordinates and by scaling with respect to the origin, we obtain the following family of solutions to \eqref{eq:U}:
\begin{equation}\label{form:U}
\epsilon^{-\frac{n-2}{2}}U_{\kappa}\big(\epsilon^{-1}(y_1-z_1,...,y_{n-1}-z_{n-1}, y_n)\big)
=\left(\frac{\epsilon}{\sum_{j=1}^{n-1}(y_j-z_j)^2+(y_n+\epsilon )^2-\epsilon^2\kappa} \right)^{\frac{n-2}{2}},
\end{equation}
where $\epsilon>0$ and $(z_1,...,z_{n-1})\in\mathbb R^{n-1}$. It is proved in \cite{chipot-fila-shafrir} that any positive solution of \eqref{eq:U} is of the form \eqref{form:U} (see also \cite{escobar4}).

A geometric interpretation of the function $U_{\kappa}$ goes as follows. The Euclidean ball is conformally equivalent to the half-space $\Rn$ by the inversion 
$$F:\mathbb{R}_+^n\to B^n\backslash\{ (0,...,0,-1)\}$$
with respect to the unit sphere in $\mathbb R^{n}$ with centre $(0,...,0,-1)$ and radius $1$. Here, $B^n$ is the Euclidean ball in $\R^{n}$ with centre $(0,...,0,-1/2)$ and radius $1/2$. The expression for $F$ is
$$F(y_1,...y_n)=\frac{(y_1,...,y_{n-1},y_n+1)}{y_1^2+...+y_{n-1}^2+(y_n+1)^2}+(0,...,0,-1)\,,$$
and its inverse mapping $F^{-1}$ has the same expression. 

When $0<\kappa<1$, the hyperbolic metric $g_{\mathbb H_\kappa^n}$ is given by
$$g_{\mathbb H_\kappa^n}=\left(\frac{1}{1-\kappa(\xi_1^2+...+\xi_{n-1}^2+(\xi_n+1)^2)}\right) ^2\delta_{\mathbb R^n}$$
on the ball of radius $1/\sqrt{\kappa}$ centred at $(0,...,0,-1)$, where $\delta_{\mathbb R^n}$ denotes the Euclidean metric.
Also denoting by $g_{\mathbb H_\kappa^n}$ its restriction to $B^n$, direct calculations show that $F$ is a conformal map and $F^*g_{\mathbb{H}_\kappa^n}= U_{\kappa}^{\frac{4}{n-2}}\delta_{\mathbb{R}^n}$ in $\Rn$.

When $\kappa\rightarrow 0$, $g_{\mathbb{H}_\kappa^n}$ becomes the Euclidean metric $\delta_{\mathbb{R}^n}$ and $F^*\delta_{\mathbb{R}^n} = U_{0}^{\frac{4}{n-2}}\delta_{\mathbb{R}^n}$ in $\Rn$ where  
$$U_{0}(y)=(y_1^2+...+y_{n-1}^2+(y_n+1)^2)^{\frac{2-n}{2}}$$ 
is a solution to \eqref{eq:U} with $\kappa=0$. This coincides with the scalar-flat case which is modelled by the Euclidean ball (see \cite[Section 2.2]{almaraz3} for a reference).

\begin{remark}\label{rmk:horo}
The classification of positive solutions in the case $\kappa\geq 1$ was also handled in \cite{chipot-fila-shafrir}, and in \cite {escobar4} under natural decay hypothesis. While \eqref{eq:U} admits no positive solutions if $\kappa>1$, the case $\kappa=1$ has the following family of solutions:
$$
W_\epsilon(y)=(2y_1+\epsilon)^{\frac{2-n}{2}},\qquad y\in\mathbb R_+^n,
$$
where $\epsilon>0$. Those functions do not satisfy the decay hypothesis in \cite{escobar4} and are such that $(\mathbb R_+^n, W_\epsilon(y)^{\frac{4}{n-2}}\delta_{\mathbb{R}^n})$ does not extend to a compact manifold as it is isometric to a domain in the hyperbolic space bounded by a horosphere. 
\end{remark}

\subsection{The linearized equation}\label{subsec:linearized}

Fix $0<\kappa< 1$.
The existence of the family of solutions (\ref{form:U}) has two important consequences. Firstly, we see that the set of solutions to equations (\ref{eq:U}) is non-compact. In particular, the set of solutions of (\ref{main:equation:1}) is non-compact when $M$ is conformally equivalent to the round hemisphere. Secondly, the functions $J_j:=\d U_{\kappa}/\d y_j$, for $j=1,...,n-1$, and $J_n:=\frac{n-2}{2}U_{\kappa}+\sum_{b=1}^{n}y^b\d U_{\kappa}/\d y^b$, are solutions to the following homogeneous linear problem:
\ba
\begin{cases}\label{linear:homog}
\Delta\Psi-n(n+2)\kappa U_{\kappa}^{\frac{4}{n-2}}\Psi=0\,,&\text{in}\:\Rn\,,
\\
\displaystyle\frac{\d\Psi}{\d y_n}+nU_{\kappa}^{\frac{2}{n-2}}\Psi=0\,,&\text{on}\:\d\Rn\,.
\end{cases}
\end{align}

As a matter of fact, a converse statement is also true as stated in the next lemma that will be used later in Section \ref{sec:blowup:estim}. 
\begin{lemma}
\label{classifLinear}
Suppose $\Psi$ is a solution of \eqref{linear:homog}.
If $\Psi(y)=O((1+|y|)^{-\alpha})$ for some $\alpha>0$, then there exist constants $c_1,...,c_{n}$ such that
$$
\Psi (y)=\sum_{j=1}^{n-1} c_j\frac{\d U_{\kappa}}{\d y_j}+c_n\Big(\frac{n-2}{2}U_{\kappa}+\sum_{b=1}^ny^b\frac{\d U_{\kappa}}{\d y^b}\Big)\,.
$$
\end{lemma}

The proof of Lemma \ref{classifLinear} will be given at the end of this subsection.
It makes use of the classification of solutions to an eigenvalue problem on geodesic hyperbolic balls as follows.

We use the hyperboloid model for the hyperbolic space $\mathbb H^n_\kappa$ of curvature $-4\kappa$, i.e.,
$$
\mathcal H_\kappa=\{z=(z_0,...,z_n)\in \mathbb R^{1,n}\,|\:\langle z,z\rangle=-r_\kappa^{-2},\,z_0>0\}
$$
where $r_{\kappa}=1/(2\sqrt{\kappa})$ and $\mathbb R^{1,n}$ is the Minkowski space with inner product 
$$
\langle z,z\rangle=-z_0^2+z_1^2+...+z_n^2.
$$
For $t_0\in\mathbb R$, let
$$
D_{t_0}=\{z\in \mathcal H_\kappa\,|\:z_0< r_\kappa\cosh{t_0}\}
$$
be the geodesic ball of radius $r_\kappa t_0$ centred at $(0,...,0,r_\kappa)$.
Observe that its boundary is given by 
$$
\partial D_{t_0}=\{z\in \mathcal H_\kappa\,|\:z_0= r_\kappa\cosh{t_0},\:z_1^2+...+z_n^2=r_\kappa^2\text{sinh}^2t_0\}.
$$
For $z=(z_0,...,z_n)\in \partial D_{t_0}$, set $r(z)=\sqrt{z_1^2+...+z_n^2}$ and let
$$
\nu(z)=-\frac{1}{r_\kappa}\Big(r(z),\frac{z_0 z_1}{r(z)},..., \frac{z_0 z_n}{r(z)}\Big)
$$
be the inward pointing unit co-normal vector to $\partial D_{t_0}$.
A direct calculation shows that if $f$ is any of the coordinate functions $z_1$, ..., $z_n$, restricted to $D_{t_0}$, then it satisfies 
\begin{equation}\label{eq:f}
    \begin{cases}
        \Delta_{\mathcal H_\kappa}f-nr_\kappa^{-2}f=0,&\text{in}\:D_{t_0},
        \\
        \displaystyle\frac{\partial f}{\partial \nu}+r_\kappa^{-1}(\text{coth}\,t_0) f=0\,,&\text{on}\:\partial D_{t_0}.
    \end{cases}
\end{equation}

\begin{lemma}\label{lemma:eigen:loid}
Let $f$ be a solution of the problem \eqref{eq:f}.
Then $f$ is a linear combination of the coordinate functions $z_1,...,z_n$, restricted to $D_{t_0}$.
\end{lemma}
\begin{proof}
Let $f$ be a solution of the problem \eqref{eq:f} and let $\Phi$ be its restriction to $\partial D_{t_0}$. 
Set $r_0=r_\kappa \text{sinh}\,t_0$.
Due to the decomposition in spherical harmonics, we write 
$$
\Phi=\sum_{\mu=0}^\infty \Phi_\mu,
$$
where each $\Phi_\mu$ satisfies
$$
\Delta_{\mathbb S_{r_0}^{n-1}}\Phi_\mu+\frac{n-1}{r_0^2}\Phi_\mu=0,
$$
and is a restriction of a homogeneous polynomial $\widetilde \Phi_\mu$ of degree $\mu$ in the coordinates $z_1$, ..., $z_n$. It is a direct computation to check that each $\widetilde\Phi_\mu$ satisfies  
$$
\Delta_{\mathcal H_\kappa}\widetilde\Phi_\mu-nr_\kappa^{-2}\widetilde\Phi_\mu=0.
$$
As $f$ and $\sum_\mu\widetilde\Phi_\mu$ satisfy the same interior equation in $D_{t_0}$ and coincide on $\partial D_{t_0}$, we have
$$
f=\sum_\mu\widetilde\Phi_\mu.
$$
Using the boundary condition in \eqref{eq:f}, we readily see that $\widetilde\Phi_\mu=0$ for any $\mu\neq 1$.
\end{proof}

\begin{proof}[Proof of Lemma \ref{classifLinear}]
Recall that the operators $L_g$ and $B_g$ satisfy the following conformal transformation laws
$$
L_{\xi^{\frac{4}{n-2}}g}(\xi^{-1}u)=\xi^{-\frac{n+2}{n-2}}L_gu
\qquad\text{and}\qquad
B_{\xi^{\frac{4}{n-2}}g}(\xi^{-1}u)=\xi^{-\frac{n}{n-2}}B_gu,
$$
for any smooth functions $\xi>0$ and $u$.
Hence, the equations \eqref{linear:homog} are equivalent to 
\begin{align*}
	\begin{cases}
		\Delta_{g_{\mathbb{H}_\kappa^n}} \overline{\Psi} -4n\kappa\overline \Psi= 0\,,&\text{in}\:\:B^n\backslash\{(0,...,0,-1)\},\\
		\displaystyle\frac{\partial \overline{\Psi}}{\partial \eta}+2\overline{\Psi}=0\,,&\text{on}\:\:\partial B^n\backslash\{(0,...,0,-1)\},
	\end{cases}
 \end{align*}
with $\overline\Psi=(U_\kappa^{-1}\Psi)\circ F^{-1}$, where $\eta$ stands for the inward unit normal vector. The hypothesis $\Psi(y)=O((1+|y|)^{-\alpha})$ implies that the singularity at $(0,...,0,-1)$ is removable so that $\overline\Psi$ solves
\begin{align*}
	\begin{cases}
		\Delta_{g_{\mathbb{H}_\kappa^n}} \overline{\Psi} -4n\kappa\overline \Psi= 0\,,&\text{in}\:\:B^n,\\
		\displaystyle\frac{\partial \overline{\Psi}}{\partial \eta}+2\overline{\Psi}=0\,,&\text{on}\:\:\partial B^n.
	\end{cases}
\end{align*}
Now choose $t_0\in\mathbb R$ such that $\text{cosh}\, t_0=2r_\kappa$. This allows us to choose an isometry $\gamma_\kappa$ from $(B^n, g_{\mathbb H_\kappa^n})$ to $D_{t_0}\subset \mathbb R^{1,n}$. Denote
$$
F_{\kappa}=\gamma_{\kappa}\circ F: \R^n_+\to D_{t_0}.
$$
Then $f=\overline{\Psi}\circ {\gamma_{\kappa}}^{-1}=(U_{\kappa}^{-1}\Psi)\circ F_{\kappa}^{-1}$ is a solution to 
\begin{align}\label{eigen:loid}
	\begin{cases}
		\Delta_{\mathcal H_\kappa}f -4n\kappa f= 0\,,&\text{in}\:\: D_{t_0},\\
		\displaystyle\frac{\partial f}{\partial \nu}+2f=0\,,&\text{on}\:\:\partial D_{t_0}.
	\end{cases}
\end{align}
In other words, $\Psi$ is a solution to  \eqref{linear:homog} if and only if $f$ solves \eqref{eigen:loid}. According to Lemma \ref{lemma:eigen:loid}, the eigenspace of \eqref{eigen:loid} is an $n$-dimensional space spanned by the restrictions of the coordinate functions $\{z_1,...,z_n\}$ to $D_{t_0}$. Therefore, the $n$-dimensional space spanned by $\{J_1,...,J_n\}$ contains all the solutions to \eqref{linear:homog} since $\{J_1,...,J_n\}$ are linearly independent solutions to \eqref{linear:homog}. This ends the proof of Lemma~\ref{classifLinear}.
\end{proof}




\section{Isolated and isolated simple blow-up points}
\label{sec:isolated}

In this section, we collect the definitions and main results on isolated and isolated simple blow-up sequences for equations \eqref{main:equation:0}. We refer the reader to \cite{han-li1} for proofs and more details.
Although commonly stated for the case $K\geq 0$ in \eqref{main:equation:0}, those results also hold for $K<0$ with some minor modifications.

\begin{definition}\label{def:blow-up}
	Let $M$ be a compact manifold with boundary and dimension $n\geq 3$, and let $\Omega\subset M$ be a domain satisfying $\Omega\cap\partial M\neq\emptyset$. Given $0<\kappa<1$, let $\{u_i>0\}_{i=1}^{\infty}$ be a sequence satisfying  
\begin{align}\label{eq:blow-up}
\begin{cases}
L_{g_i}u_i-n(n-2)\kappa u_i^{\frac{n+2}{n-2}}=0,&\text{in}\:\Omega,
\\
B_{g_i}u_i+(n-2)u_i^{\frac{n}{n-2}}=0,&\text{on}\:\Omega\cap \partial M,
\end{cases}
\end{align}
where $g_i$ is a metric on $\Omega$.
We say that $x_0\in \Omega\cap\partial M$ is a {\it{blow-up point}} for $\{u_i\}$ if there is a sequence $\{x_i\}\subset \Omega\cap \partial M$ such that, as $i\to\infty$,

	(1) $x_i\to x_0$;
	
	(2) $u_i(x_i)\to\infty$; 
	
	(3) $x_i$ is a local maximum of $u_i$.\\
	The sequence $\{x_i\}$ is called a {\it{blow-up sequence}}. Briefly we say that $x_i\to x_0$ is a blow-up point for $\{u_i\}$, meaning that $\{x_i\}$ is a blow-up sequence converging to the point $x_0$. Throughout the paper we will assume $g_i\to g_0$ in $C^2(\Omega)$. 
\end{definition}

\noindent
{\bf{Agreement.}} If $x_i\to x_0$ is a blow-up point, we use Fermi coordinates $$\fermi$$ centred at $x_i$ and work in $B^+_{\delta}(0)\subset\Rn$, for some small $\delta>0$.
\begin{definition}\label{def:isolado}
	We say that a blow-up point $x_i\to x_0$ is an {\it{isolated}} blow-up point for $\{u_i\}$ if there exist $\delta,C>0$ such that 
	\begin{equation}\notag
	u_i(x)\leq Cd_{g_i}(x,x_i)^{-\frac{n-2}{2}}\,,\:\:\:\:\text{for all}\: x\in M\backslash \{x_i\}\,,\:d_{g_i}(x,x_i)< \delta\,.
	\end{equation}
\end{definition}
The above definition is equivalent to
	\begin{equation}\label{des:isolado}
	u_i(\psi_i(z))\leq C|z|^{-\frac{n-2}{2}}\,,\:\:\:\:\text{for all}\: z\in B^+_{\delta}(0)\backslash\{0\}\,.
	\end{equation}
	This estimate is invariant under re-scaling. This follows from the fact that if $v_i(y)=s^{\frac{n-2}{2}}u_i(\psi_i(sy))$, then
	$$
	u_i(\psi_i(z))\leq C|z|^{-\frac{n-2}{2}}\Longleftrightarrow v_i(y)\leq C|y|^{-\frac{n-2}{2}}\,,
	$$
	where $z=sy$.
	
	Harnack inequalities give the following:
\begin{lemma}\label{Harnack}
	Let $x_i\to x_0$ be an isolated blow-up point and $\delta$ as in Definition \ref{def:isolado}. 
	Then there exists $C>0$ such that for any $0<s<\frac{\delta}{3}$ we have
	\begin{equation}\notag
	\max_{B_{2s}^+(0)\backslash B_{s/2}^+(0)} (u_i\circ\psi_i)\leq C\min_{B_{2s}^+(0)\backslash B_{s/2}^+(0)} (u_i\circ\psi_i)\,.
	\end{equation}
\end{lemma}

As a consequence, we state the next proposition which says that, in the case of an isolated blow-up point, the sequence $\{u_i\}$, when renormalized, converges to the standard solutions \eqref{eq:Uk}.

\begin{proposition}\label{form:bolha}
	Let $x_i\to x_0$ be an isolated blow-up point. Given $R_i\to\infty$ and $0<\b_i\to 0$, after choosing subsequences, we have  
	\\\\
	(1) $\big| u_i(x_i)^{-1}u_{i}(\psi_i(u_i(x_i)^{-\frac{2}{n-2}}y))-\lambda^{\frac{n-2}{2}}U_{\kappa}(\lambda y) \big|_{C^2(B^+_{R_i}(0))}<\b_i,$ where $\lambda=1-\kappa$;
	\\\\
	(2) $\displaystyle\lim_{i\to\infty}\frac{R_i}{\log u_i(x_i)}= 0$.
\end{proposition}

The set of blow-up points is handled in the next proposition.
\begin{proposition}\label{conj:isolados}
	Given $0<\kappa<1$, small $\b>0$ and large $R>0$, there exist constants $C_0, C_1>0$, depending only on $\kappa$, $\b$, $R$ and $(M^n,g)$, such that if $u>0$ is solution of \eqref{main:equation:1} with  
$$
\max_{M} u\geq C_0,
$$
then there exist $x_1,...,x_N\in\d M$, $N=N(u)\geq 1$, local maxima of $u$, such that:
	\\\\
	(1) If $r_j=Ru(x_j)^{-\frac{2}{n-2}}$ for $j=1,...,N$,  then $\{B_{r_j}(x_j)\subset M\}_{j=1}^{N}$ is a disjoint collection, where $B_{r_j}(x_j)$ is the metric ball.
	\\\\
	(2) For each $j=1,...,N$, 
	$\:\:\:\:
	\left|u(x_j)^{-1}u(\psi_j(z))-\lambda^{\frac{n-2}{2}}U_{\kappa}(\lambda u(x_j)^{\frac{2}{n-2}}z)\right|_{C^2(B^+_{2r_j}(0))}<\b
	$,
	\\
	where we are using Fermi coordinates $\psi_j:B^+_{2r_j}(0)\to M$ centred at $x_j$ and setting $\lambda=1-\kappa$.
	\\\\
	(3) We have
	$$
	u(x)\,d_{g}(x,\{x_1,...,x_N\})^{\frac{n-2}{2}}\leq C_1\,,\:\:\:\text{for all}\: x\in M\,,
	$$
	$$
	u(x_j)\,d_{g}(x_j,x_l)^{\frac{n-2}{2}}\geq C_0\,,\:\:\:\:\text{for any}\: j\neq l\,,\:j,l=1,...,N\,.
	$$ 
\end{proposition}

The proof of Proposition \ref{conj:isolados} is similar to \cite[Proposition 1.1]{han-li1} and is based on the following lemma:

\begin{lemma}\label{compact:set:lemma}
Given $0<\kappa<1$ and $R,\beta>0$, there exists $C_0=C_0(\kappa, R,\beta)>0$ such that if $u>0$ is a solution of  \eqref{main:equation:1} and $S\subset M$ is a compact set, we have the following:

If $\max_{x\in M\backslash S}\left(u(x)d_g(x,S)^{\frac{n-2}{2}}\right)\geq C_0$, then there exists $x_0\in\partial M\backslash (S\cap \partial M)$ local maximum of $u$, such that 
\begin{equation}\label{eq:approx:cpct}
\left|u(x_0)^{-1}u(\psi(z))-\lambda^{\frac{n-2}{2}}U_{\kappa}(\lambda u(x_0)^{\frac{2}{n-2}}z)\right|_{C^2(B^+_{2r_0}(0))}<\b
\end{equation}
where $r_0=Ru(x_0)^{-\frac{2}{n-2}}$ and $\lambda=1-\kappa$. Here, we are using Fermi coordinates $\psi:B^+_{r_0}(0)\to M$ centred at $x_0$.  If $S$ is the empty set $\emptyset$, we define $d_g(x,S)=1$.
\end{lemma}

\begin{proof}
Suppose by contradiction there exist $R,\beta>0$ satisfying the following:
for all $C_0>0$ there exist $u>0$ solution of  \eqref{main:equation:1} and a compact set $S\subset M$ such that 
$$
\max_{x\in M\backslash S}\left(u(x)d_g(x,S)^{\frac{n-2}{2}}\right)\geq C_0
$$
holds and no such point $x_0$ exists. Hence, we can suppose that there are sequences $u_i>0$ solving \eqref{main:equation:1}, $x_i'\in M$, and $S_i\subset M$ compact sets such that 
$$
w_i(x_i')=\max_{x\in M\backslash S_i}w_i(x)\to\infty, \qquad\text{as}\:\:i\to\infty,
$$
where $w_i(x):=u_i(x)d_g(x,S_i)^{\frac{n-2}{2}}$. Observe that $u(x'_i)\to\infty$ as $i\to\infty$.

The rest of the proof is rather standard by now and we will only sketch it, pointing out the modifications. 
We use local coordinates centred at $x_i'$ and rescale the $u_i$ by $u(x'_i)$ to obtain a sequence of functions $\{v_i\}$ which is uniformly bounded in compact sets and satisfies $v_i(0)=1$. Assuming $\{v_i\}$ converges to a limit function $v$, we first rule out the possibility that 
$$
T_i:=u(x'_i)d_g(x_i',\partial M)^{\frac{n-2}{2}}\to\infty.
$$ 
Assuming by contradiction it happens, then $v\geq 0$ would satisfy the equation
$\Delta v-n(n-2)\kappa v^{\frac{n+2}{n-2}}=0$ on $\mathbb R^n$. This would imply $v\equiv 0$ (see \cite[Lemma 2]{brezis}) which contradicts $v(0)=1$.
So, we can assume that $T_i$ is convergent as $i\to\infty$. We proceed as in \cite[Lemma 1.1]{han-li1} to obtain a point $x_0\in\partial M$  such that \eqref{eq:approx:cpct} holds. This is a contradiction and proves the lemma.

\end{proof}

We now introduce the notion of an isolated simple blow-up point. If $x_i\to x_0$ is an isolated blow-up point for $\{u_i\}$,  for $0<r<\delta$, set 
$$\bar{u}_i(r)=\frac{2}{\sigma_{n-1}r^{n-1}}\int_{S^+_r(0)}(u_i\circ\psi_i)d\sigma_r
\:\:\:\:\text{and}\:\:\:\:w_i(r)=r^{\frac{n-2}{2}}\bar{u}_i(r)\,.$$ 
Note that the definition of $w_i$ is invariant under re-scaling. More precisely, if $v_i(y)=s^{\frac{n-2}{2}}u_i(\psi_i(sy))$, then
$r^{\frac{n-2}{2}}\bar{v}_i(r)=(sr)^{\frac{n-2}{2}}\bar{u}_i(sr)$.
\begin{definition}\label{def:simples}
	An isolated blow-up point $x_i\to x_0$ for $\{u_i\}$ is {\it{simple}} if there exists $\delta>0$ such that $w_i$ has exactly one critical point in the interval $(0,\delta)$.
\end{definition}
\begin{remark}\label{rk:def:equiv} Let $x_i\to x_0$ be an isolated blow-up point and $R_i\to\infty$. Using Proposition \ref{form:bolha} it is not difficult to see that, choosing a subsequence, 
$$
r\mapsto r^{\frac{n-2}{2}}\bar{u}_i(r)
$$ 
has exactly one critical point in the interval $(0,r_i)$, where $r_i=R_i u_i(x_i)^{-\frac{2}{n-2}} \to 0$. Moreover, its derivative is negative right after the critical point. Hence, if $x_i\to x_0$ is isolated simple then there exists $\delta>0$ such that $w_i'(r)<0$ for all $r\in [r_i,\delta)$. 
\end{remark}

A basic result for isolated simple blow-up points is stated as follows:
\begin{proposition}\label{estim:simples}
	Let $x_i\to x_0$ be an isolated simple blow-up point for $\{u_i\}$.  Then there exist $C,\delta>0$ such that
	\\\\
	(a) $u_i(x_i)u_i(\psi_i(z))\leq C|z|^{2-n}$\:\:\: for all $z\in B^+_{\delta}(0)\backslash\{0\}$;
	\\\\
	(b) $u_i(x_i)u_i(\psi_i(z))\geq C^{-1}G_i(z)$\:\:\: for all $z\in B^+_{\delta}(0)\backslash B^+_{r_i}(0)$, where $G_i$ is the Green's function so that:
	
	\begin{align}
	\begin{cases}
	L_{g_i}G_i=0, &\text{in}\; B_{\delta}^+(0)\backslash\{0\},\notag
	\\
	G_i=0, &\text{on}\; S_{\delta}^+(0),\notag
	\\
	B_{g_i}G_i=0, &\text{on}\;D_{\delta}(0)\backslash\{0\}\notag 
	\end{cases}
	\end{align}
	and $|z|^{n-2}G_i(z)\to 1$, as $|z|\to 0$. Here, $r_i$ is defined as in Remark \ref{rk:def:equiv}.
\end{proposition}

\bp
The proof of Proposition \ref{estim:simples} follows from standard arguments and we only sketch it here without providing all details. See for example the proof of Proposition 1.4 in \cite{han-li1} or Proposition 4.3 in \cite{almaraz3} for references. Let us use Fermi coordinates $\psi_i$ centred at $x_i$ and work with $B_{\delta}^+\subset \R^n_+$ omitting the symbol $\psi_i$ for simplicity. It follows directly from Proposition \ref{form:bolha} that there are large $C>0$ and small $\rho>0$ such that 
\begin{equation}\label{estim:rhp}
	u_i(x_i)^{1-\frac{\rho}{n-2}}u_i(z)\leq C|z|^{2-n+\rho},
\end{equation}
for any $z\in B^+_{r_i}$, where $r_i$ is defined as in Remark \ref{rk:def:equiv}. 
We would like to extend \eqref{estim:rhp} to $B_{\delta}^+ \backslash B_{r_i}^+$ for some $\delta>0$. Arguing as in \cite[Lemma 2.2]{han-li1}, we have to construct appropriate comparison functions $\Psi_i$ in order to apply the maximum principle in \cite[Lemma A.2]{han-li1}. If we set $$
L_i=L_{g_i}-n(n-2)\kappa u_i^{\frac{4}{n-2}}\qquad \text{and}\qquad B_i=B_{g_i}+(n-2)u_i^{\frac{2}{n-2}},
$$
then $u_i$ satisfies 
	\begin{align*}
	\begin{cases}
		L_{i}u_i=0, &\text{in}\; B_{\delta}^+\backslash B_{r_i}^+,\notag
		\\
		B_{i}u_i=0, &\text{on}\;D_{\delta}\backslash D_{r_i}.\notag 
	\end{cases}
\end{align*}
The non-negative test functions in our situation are chosen to be $$\Psi_i(z)=C\(u_i(x_i)^{\frac{\rho}{n-2}-1}\Phi_{i,n-2-\rho}(z)+\eta_i\Phi_{i,\rho}(z)\),$$ where $\Phi_{i,\nu}(z)=|z|^{-\nu}-\epsilon_0|z|^{-\nu-1}z_n$ for some $\epsilon_0>0$ and $\eta_i=\max_{S_{\delta}^+}u_i$. Here, $C>0$ is chosen such that $u_i\leq \Psi_i$ on $S_{\delta}^+\cup S_{r_i}^+$. By choosing $\epsilon_0$ small, direct calculations give 
	\begin{align*}
	\begin{cases}
		L_{i}(\Psi_i-u_i)\leq 0, &\text{in}\; B_{\delta}^+\backslash B_{r_i}^+,\notag
		\\
		B_{i}(\Psi_i-u_i)\leq0, &\text{on}\;D_{\delta}\backslash D_{r_i}.\notag 
	\end{cases}
\end{align*}
It then follows from Lemma A.2 in \cite{han-li1} that $u_i\leq \Psi_i$ in $B_{\delta}^+\backslash B_{r_i}^+$. This extends \eqref{estim:rhp} to $B_{\delta}^+\backslash B_{r_i}^+$ for $\delta>0$. 

Now we proceed as in \cite[Proposition 4.3]{almaraz3}. A change of scale using Proposition \ref{form:bolha} for the region $D_{r_i}$ and \eqref{estim:rhp} for the region $D_{\delta}\backslash D_{r_i}$ shows that
\begin{equation}\label{claim1}
\int_{D_{\delta}}u_i^{\frac{n}{n-2}}\leq CM_i^{-1}.
\end{equation} 
Here and in the rest of the proof, we are omitting the volume and area elements.

We claim that there exists $\sigma_1>0$ such that for all $0<\sigma<\sigma_1$ there exists $C'_\sigma$ satisfying
\begin{equation}\label{claim2}
u_i(x_i)u_i(z)\leq C'_\sigma
\end{equation}
for all $z\in S^+_{\sigma}$. Fix $\sigma\in (0,\sigma_1)$ and choose any $x_{\sigma}\in S^+_{\sigma}$. If we set $w_i=u_i(x_\sigma)^{-1}u_i$, then $w_i$ satisfies
\begin{equation}\label{eq:wi}
\begin{cases}
\Delta_{g_i} w_i-n(n-2)\kappa u_i(x_\sigma)^{\frac{4}{n-2}}w_i^{\frac{n+2}{n-2}}=0, &\text{in}\;B_{\sigma}^+,
\\
\partial_n w_i+(n-2)u_i(x_\sigma)^{\frac{2}{n-2}}w_i^{\frac{n}{n-2}}=0,&\text{on}\;D_{\sigma}.
\end{cases}
\end{equation}
By the Harnack inequality in \cite[Lemma A.1]{han-li1}, for each $\beta>0$ there exists $C_\beta$ such that 
$$
C_\beta^{-1}\leq w_i(z)\leq C_\beta
$$
if $|z|>\beta$. Observe that \eqref{estim:rhp} implies that $u_i(x_\sigma)\to 0$ as $i\to\infty$. Hence, we can suppose that $w_i\to w$ in $C_{loc}^2(B_\sigma\backslash \{0\})$, for some $w>0$ satisfying
\begin{equation*}
\begin{cases}
\Delta_{g_0} w=0, &\text{in}\;B_{\sigma}^+\backslash\{0\},
\\
\partial_n w=0,&\text{on}\;D_{\sigma}\backslash\{0\}.
\end{cases}
\end{equation*}
It follows from elliptic linear theory that $w(z)=aG(z)+b(z)$ for $z\in B^+_\sigma\backslash\{0\}$, where $a\geq 0$. Here, $G$ is the Green's function so that  
\begin{equation*}
\begin{cases}
\Delta_{g_0} G=0, &\text{in}\;B_{\sigma}^+\backslash\{0\},
\\
\partial_n G=0,&\text{on}\;D_{\sigma}\backslash\{0\},
\\
G=0,&\text{on}\;S^+_{\sigma},
\\
\lim_{|z|\to 0}|z|^{n-2}G(z)=1,
\end{cases}
\end{equation*}
and $b$ satisfies 
\begin{equation*}
\begin{cases}
\Delta_{g_0} b=0, &\text{in}\;B_{\sigma}^+,
\\
\partial_n b=0,&\text{on}\;D_{\sigma}.
\end{cases}
\end{equation*}
Using the definition of isolated simple blow-up, one can prove that the origin is a non-removable singularity for $w$, so that $a>0$. Thus, there exists $c_1>0$ such that 
$$
-\int_{S^+_\sigma}\frac{\partial w}{\partial r}>c_1.
$$
Here and in what follows, we are omitting the volume and area elements.
Integrating by parts the first equation of \eqref{eq:wi}, we obtain
\begin{align*}
n(n-2)\kappa u_i(x_\sigma)^{-1}\int_{B_\sigma^+}u_i^{\frac{n+2}{n-2}}
&=\int_{B_\sigma^+}\Delta_{g_i}w_i
=\int_{S_\sigma^+}\frac{\partial w_i}{\partial r}-\int_{D_\sigma}\partial_n w_i
\\
&=\int_{S_\sigma^+}\left(\frac{\partial w}{\partial r}+o_i(1)\right)+(n-2)u_i(x_\sigma)^{-1}\int_{D_\sigma}u_i^{\frac{n}{n-2}}
\\
&\leq -\frac{c_1}{2}+(n-2)u_i(x_\sigma)^{-1}\int_{D_\sigma}u_i^{\frac{n}{n-2}}.
\end{align*}
As the left-hand side above is non-negative, estimate \eqref{claim2} follows from inequality \eqref{claim1}.

Once we have proved  \eqref{claim2}, item (a) follows from a contradiction argument.
Item (b) is an application of the maximum principle in \cite[Lemma A.2]{han-li1}.
\ep


\section{Blow-up estimates in dimension three}\label{sec:blowup:estim}
In this section, we give a point-wise estimate for a blow-up sequence in a neighbourhood of an isolated simple blow-up point. 

\vspace{0.2cm}
\noindent
{\bf{Agreement.}}
Throughout this section, we assume that $n=3$ as refined blow-up estimates in Proposition~\ref{estim:blowup:compl} are not required in the case of locally conformally flat manifolds with umbilic boundary.
Let  $x_i\to x_0$ be an isolated simple blow-up point for the sequence $\{u_i\}$ of positive solutions to the equations \eqref{eq:blow-up} with $h_{g_i}=0$ and let $\psi_i: B^+_{\delta'}\to M$ denote Fermi coordinates centred at $x_i$. Set $\ei=u_i(x_i)^{-2}$, which goes to zero as $i\to\infty$. Recall that we are assuming that $0<\kappa<1$ and $g_i\to g_0$. 
\vspace{0.2cm}

Let $r\mapsto 0\leq\chi(r)\leq 1$ be a smooth cut-off function such that $\chi(r)\equiv 1$ for $0\leq r\leq \delta'$ and $\chi(r)\equiv 0$ for $r>2\delta'$. 
We set $\chi_{\e}(r)=\chi(\e r)$.
Thus,  $\chi_{\e}(r)\equiv 1$ for $0\leq r\leq \delta'\e^{-1}$ and $\chi_{\e}(r)\equiv 0$ for $r>2\delta'\e^{-1}$.

The next proposition, although stated for dimension three, also holds for any $n\geq 3$ with some obvious modifications.

\begin{proposition}\label{Linearized} For every $i$ there is a solution $\phi_i$ of 
	\begin{align}
		\begin{cases}\label{linear:8}
			\Delta\phi_{i}(y)-15\kappa U_{\kappa}^{4}\phi_{i}(y)=-2\chi_{\ei}(|y|)\lambda^{-1}\ei y_3 \pi_{jl}(x_i)\displaystyle\frac{\d^2 U_{\kappa}(y)}{\d y_j\d y_l}\,,&\text{for}\:y\in\R^3_+\,,
			\\
			\displaystyle\frac{\d\phi}{\d y_3} (\bar y)+3 U_{\kappa}^{2}\phi_{i}(\bar{y})=0\,,&\text{for}\:\bar{y}\in\d\R^3_+\,,
		\end{cases}
	\end{align} 
	where $\Delta$ stands for the Euclidean Laplacian and $\lambda=1-\kappa$, satisfying 
	\begin{equation}\label{estim:phi'}
		|\nabla^s\phi_i|(y)\leq C\ei |\pi_{jl}( x_i)|(1+|y|)^{-s}\,,\:\:\:\:\text{for}\: y\in\R^3_+\,,\,s=0,1 \:\text{or}\:2,
	\end{equation} 
	and
	\begin{equation}\label{hip:phi}
		\phi_i(0)=\frac{\d\phi_i}{\d y_1}(0)=\frac{\d\phi_{i}}{\d y_{2}}(0)=0.
	\end{equation}
\end{proposition}

\bp
The case $\kappa=0$ is Proposition 5.1 of \cite{almaraz3}. 
The proof for $0<\kappa<1$ follows the same lines as of that proposition with some minor modifications and will be sketched here. It only requires $\{x_i\}$ to be a sequence of points on $\partial M$, not necessarily a blow up sequence, 
and it makes use of the fact that $h_g(x_i)=0$ for every $i$.
It strongly relies on the conformal equivalence $F_{\kappa}:\mathbb R^3_+\to D_{t_0}$ defined in Section \ref{sec:pre}.

Define, for each $i$,
$$
f_i(F_{\kappa}(y))=-2\chi_{\ei}(|y|)\lambda^{-1}\ei y_3 \pi_{jl}( x_i)\displaystyle\frac{\d^2 U_{\kappa}(y)}{\d y_j\d y_l}U_{\kappa}^{-5}(y),\:\:y\in\R^3_+.
$$
According to Lemma \ref{classifLinear}, for each $a=1,2,3$ there exist constants $c_b$, $b=1,2,3$, such that 
$$
z_a\circ F_{\kappa}=\sum_{b=1}^{3}c_bU_{\kappa}^{-1}J_b,
$$
where the $J_b$ are as in Subsection \ref{subsec:linearized}. By symmetry arguments,
$$
\int_{D_{t_0}}z_1f_i=\int_{D_{t_0}}z_2f_i=\int_{D_{t_0}}z_3f_i=0,
$$
since ${\text{tr}}(\pi_{jl}(x_i))=2h_{g_i}(x_i)=0.$
As the coordinate functions $z_1,z_2$ and $z_3$ generate the space of solutions of \eqref{eigen:loid}, in particular for each $i$ one can find a smooth solution $\overline\phi_{\epsilon_i}$ of  
\begin{equation*}
	\begin{cases}
		\Delta_{\mathcal H_\kappa} \overline{\phi}_{\epsilon_i} -12\kappa\overline \phi_{\epsilon_i}= f_i\,,&\text{in}\:D_{t_0}\,,\\
		\displaystyle\frac{\partial \overline{\phi}_{\epsilon_i}}{\partial \nu}+2\overline{\phi}_{\epsilon_i}=0\,,&\text{on}\:\partial D_{t_0}\,,
	\end{cases}
\end{equation*}
which is $L^2(D_{t_0})$-orthogonal to $z_1, z_2$ and $z_3$. Here, as in Section \ref{sec:pre}, $\nu$ stands for the unit co-normal vector to $\partial D_{t_0}$ pointing inwards. We then consider the Green's function $G(z,w)$ on $D_{t_0}$ such that
\begin{equation*}
	\begin{cases}
		\Delta_{\mathcal H_\kappa} G -12\kappa G= \alpha(z_1w_1+z_2w_2+z_3w_3)\,,&\text{in}\:D_{t_0}\,,\\
		\displaystyle\frac{\partial G}{\partial \nu}+2G=0\,,&\text{on}\:\partial D_{t_0}\,,
	\end{cases}
\end{equation*}
where $\alpha=\|z_1\|_{L^2(D_{t_0})}^{-2}=\|z_2\|_{L^2(D_{t_0})}^{-2}=\|z_3\|_{L^2(D_{t_0})}^{-2}$.

Observe that $\overline\phi_i=U_\kappa\,\overline\phi_{\epsilon_i}\circ F_{\kappa}$ satisfies the required equations \eqref{linear:8} and a Green's representation formula argument gives the required estimates \eqref{estim:phi'}. Finally we choose
$$
\phi_i=\overline\phi_i+c_{1,i}J_1+c_{2,i}J_2+c_{3,i}J_3.
$$
Clearly, the coefficients $c_{1,i}, c_{2,i}$ and $c_{3,i}$ can be chosen so that the equations \eqref{hip:phi} hold, and $\phi_i$ satisfies \eqref{linear:8} because $J_1$, $J_2$ and $J_3$ are solutions to \eqref{linear:homog}. 
As $|c_{a,i}|\leq C\epsilon_i|\pi_{jl}(x_i)|$, for $a=1,2,3$, it is easy to see that each $\phi_i$ satisfies  \eqref{estim:phi'}.\ep

Now we are ready to obtain some refined estimates at isolated simple blow-up points. 
Define $v_i(y)=\ei^{1/2}u_i(\psi_i(\ei y))$ for $y\in \Beilinha$. We know that $v_i$ satisfies 
\begin{align}\label{eq:vi'}
	\begin{cases}
		L_{\hat{g}_i}v_i-3\kappa v_i^{5}=0,&\text{in}\:\Beilinha,
		\\
		B_{\hat{g}_i}v_i+v_i^{3}=0,&\text{on}\:\Deilinha,
	\end{cases}
\end{align}
where $\hat{g}_i$ is the metric with coefficients $(\hat{g}_i)_{jl}(y)=(g_i)_{jl}(\psi_i(\ei y))$. The following proposition strongly relies on the assumption $n=3$:

\begin{proposition}\label{estim:blowup:compl}
There exist $\delta, C>0$ such that, for $|y|\leq\delta\ei^{-1}$,
	\begin{equation}\label{eq:ei}
		|\nabla^s(v_i-\widetilde{U}_{\kappa}-\widetilde{\phi}_i)|(y)\leq C\ei (1+|y|)^{-s}\,,
	\end{equation}
for $s=0,1,2$, where $\widetilde{U}_{\kappa}=\lambda^{1/2}U_{\kappa}(\lambda y)$, $\widetilde{\phi}_{i}=\lambda^{1/2}\phi_{i}(\lambda y)$ and $\lambda=1-\kappa$.
\end{proposition}

\begin{proof}
We will prove for $s=0$. The estimates with $s=1,2$ follow from elliptic theory. It follows from Proposition~\ref{Linearized} and scaling properties that $\tilde{\phi}_i$ satisfies 
\begin{align}
		\begin{cases}\label{eq:phitilde}
			\Delta\widetilde\phi_{i}(y)-15\kappa \widetilde U_{\kappa}^{4}\widetilde\phi_{i}(y)=-2\chi_{\ei}(|y|)\ei y_3 \pi_{jl}(x_i)\displaystyle\frac{\d^2 \widetilde U_{\kappa}(y)}{\d y_j\d y_l}\,,&\text{for}\:y\in\R^3_+\,,
			\\
			\displaystyle\frac{\d\widetilde\phi}{\d y_3} (\bar y)+3 \widetilde U_{\kappa}^{2}\widetilde\phi_{i}(\bar{y})=0\,,&\text{for}\:\bar{y}\in\d\R^3_+\,,
		\end{cases}
	\end{align} 
and
\begin{equation}\label{hip:phitilde}
	\widetilde\phi_i(0)=\frac{\d\widetilde\phi_i}{\d y_1}(0)=\frac{\d\widetilde\phi_{i}}{\d y_{2}}(0)=0.
\end{equation}
Moreover, the estimates \eqref{estim:phi'} also hold for $\widetilde\phi_i$.

We consider $\delta<\delta'$ to be chosen later and set 
	$$\Lambda_i=\max_{|y|\leq \delta\ei^{-1}} |v_i-\widetilde U_{\kappa}-\widetilde{\phi}_i|(y)=|v_i-\widetilde U_{\kappa}-\widetilde{\phi}_i|(y_i)\,,$$ 
	for some $|y_i|\leq \delta\ei^{-1}$. From Propositions \ref{form:bolha} and \ref{estim:simples},
	we know that $v_i(y)\leq C\widetilde U_{\kappa}(y)$ for $|y|\leq \delta\ei^{-1}$. Hence, if there exists $c>0$ such that $|y_i|\geq c\ei^{-1}$, then
	$$|v_i-\widetilde U_{\kappa}-\widetilde{\phi}_i|(y_i)\leq C\,|y_i|^{-1}\leq C\,\ei.$$
	This already implies the inequality \eqref{eq:ei} for $|y|\leq \delta\ei^{-1}$. Hence, we can suppose that $|y_i|\leq \delta\ei^{-1}/2$. 
	
	Suppose, by contradiction, the result is false. 
	Then, choosing a subsequence if necessary, we can suppose that
	\begin{equation}
		\label{hipLambda}
		\lim_{i\to\infty}\Lambda_i^{-1}\ei=0\,.
	\end{equation}
	Define
	$$w_i(y)=\Lambda_i^{-1}(v_i-\widetilde U_{\kappa}-\widetilde \phi_i)(y)\,,\:\:\:\:\text{for}\:\: |y|\leq \delta\ei^{-1}\,.$$
	By the equations \eqref{eq:U} and \eqref{eq:vi'}, $w_i$ satisfies
	\begin{equation}\label{wi}
		\begin{cases}
			L_{\hat{g}_i}w_i+b_i w_i=Q_i\,,&\text{in}\:\Bei\,,\\
			B_{\hat{g}_i}w_i+\bar b_i w_i=\overline{Q}_i\,,&\text{on}\:\Dei\,,
		\end{cases}
	\end{equation}
	where 
	$$b_i=-3\kappa\frac{v_i^{5}-(\widetilde U_{\kappa}+\widetilde \phi_i)^{5}}{v_i-(\widetilde U_{\kappa}+\widetilde \phi_i)},
	\qquad\bar b_i=\frac{v_i^{3}-(\widetilde U_{\kappa}+\widetilde \phi_i)^{3}}{v_i-(\widetilde U_{\kappa}+\widetilde \phi_i)},$$
	$$Q_i=\Lambda_i^{-1}\big\{3\kappa (\widetilde U_{\kappa}+\widetilde \phi_i)^5-3\kappa \widetilde U_{\kappa}^5-15\kappa \widetilde U_{\kappa}^4\widetilde \phi_i-(L_{\hat g_i}-\Delta)(\widetilde U_\kappa+\widetilde \phi_i)+2\chi_{\ei}(|y|)\ei \pi_{jl}( x_i)y_3(\d_j\d_l \widetilde U_\kappa)(y)\big\},$$
	\begin{align*}
		\overline{Q}_i&=-\Lambda_i^{-1}\left\{(\widetilde U_{\kappa}+\widetilde \phi_i)^{3}
		-\widetilde U_{\kappa}^{3}-3\widetilde U_{\kappa}^2\widetilde \phi_i+\big(B_{\hat g_i}-\frac{\partial}{\partial y_3}\big)(\widetilde U_{\kappa}+\widetilde \phi_i)\right\}
		\\
		&=-\Lambda_i^{-1}\left\{(\widetilde U_{\kappa}+\widetilde \phi_i)^{3}
		-\widetilde U_{\kappa}^{3}-3\widetilde U_{\kappa}^2\widetilde \phi_i \right\}.
	\end{align*}
	Here, the last equality holds because we are using Fermi coordinates and $h_{\hat g_i}=0$ as a consequence of $h_{g_i}=0$.

	Observe that, for $v=\widetilde U_{\kappa}+\widetilde \phi_i$,
	\begin{align}
		(L_{\hat{g}_i}-\Delta)v(y)&=(\hat{g}_i^{jl}-\delta^{jl})(y)\d_j\d_l v(y)
		+(\d_j\hat{g}_i^{jl})(y)\d_l v(y)\notag
		\\
		&\hspace{2cm}-\frac{1}{8}R_{\hat{g}_i}(y)v(y)
		+\frac{\d_j \sqrt{\det \hat{g}_i}}{\sqrt{\det \hat{g}_i}}\hat{g}_i^{jl}(y)\d_l v(y)\notag
		\\
		&=2\ei y_3\pi_{jl}( x_i)\d_j\d_l v(y)+O(\ei^2(1+|y|)^{-1})\,.\notag
	\end{align}
	Hence,
	\begin{align}\label{Qi}
		Q_i(y)=O\left(\Lambda_i^{-1}\ei^2(1+|y|)^{-3}\right)+O\left(\Lambda_i^{-1}\ei^2(1+|y|)^{-1}\right)\,,
	\end{align}
	and
	\begin{equation}\label{barQi}
		\overline{Q}_i(\bar{y})= O\left(\Lambda_i^{-1}\ei^2(1+|\bar{y}|)^{-1}\right)\,.
	\end{equation}
	Moreover,
	\begin{equation}
		\label{lim:bi}
		b_i\to -15\kappa \widetilde U_\kappa^4\,,\:\bar b_i\to 3\widetilde U_\kappa^2\,,\:\text{as}\:i\to\infty,\:\:\text{in}\: C^2_{loc}(\mathbb R^3)\,,
	\end{equation}
	and
	\begin{equation}
		\label{estim:bi}
		b_i(y)\leq C(1+|y|)^{-4}\,,\:\:\bar b_i(y)\leq C(1+|y|)^{-2}\,,\:\:\:\:\text{for}\:\: |y|\leq\delta\ei^{-1}\,. 
	\end{equation}
	
	Since $|w_i|\leq |w_i(y_i)|=1$, we can use standard elliptic estimates to conclude that $w_i\to w$, in $C_{loc}^2(\mathbb{R}_+^3)$, for some function $w$, choosing a subsequence if necessary. From the identities \eqref{hipLambda}, \eqref{wi}, \eqref{Qi}, \eqref{barQi} and  \eqref{lim:bi}, we see that $w$ satisfies
	\begin{equation}\label{w}
		\begin{cases}
			\Delta w-15\kappa\widetilde  U_{\kappa}^4w=0\,,&\text{in}\:\mathbb{R}_+^3\,,\\
			\displaystyle\frac{\d w}{\d y_3}+3\widetilde U_{\kappa}^{2}w=0\,,&\text{on}\:\partial\mathbb{R}_+^3\,.
		\end{cases}
	\end{equation}

	\bigskip
	\noindent
	{\it{Claim.}} $\:\:w(y)=O((1+|y|)^{-1})$, for $y\in \Rn$.
	
	\vspace{0.2cm}
	Choosing $\delta>0$ sufficiently small, we can consider the Green's function $G_i$ for the conformal Laplacian $L_{\hat{g}_i}$ in $\Bei$ subject to the boundary conditions $B_{\hat{g}_i} G_i=0$ on $\Dei$ and $G_i=0$ on $\Sei$. Let $\eta_i$ be the inward unit normal vector to $\Sei$. Then the Green's formula gives
	\begin{align}\label{wG}
		w_i(y)&=\int_{\Bei}G_i(\xi,y)\left(b_i(\xi)w_i(\xi)-Q_i(\xi)\right) \,dv_{\hat{g}}(\xi)
		+\int_{\Sei}\frac{\partial G_i}{\partial\eta_i}(\xi,y)w_i(\xi)\,d\sigma_{\hat{g}}(\xi)\notag
		\\
		&\hspace{1cm}
		+\int_{\Dei}G_i(\xi,y)\left(\bar b_i(\xi)w_i(\xi)-\overline{Q}_i(\xi)\right)\,d\sigma_{\hat{g}}(\xi)\,.
	\end{align}
	Using the estimates \eqref{Qi}, \eqref{barQi} and \eqref{estim:bi}  in the equation \eqref{wG}, we obtain
	\begin{align}
		|w_i(y)|
		\leq &\:C\int_{\Bei}|\xi-y|^{-1}(1+|\xi|)^{-4}d\xi
		+C\Lambda_i^{-1}\ei^2\int_{\Bei}|\xi-y|^{-1}(1+|\xi|)^{-1}d\xi\notag
		\\
		&+C\int_{\Dei}|\bar{\xi}-y|^{-1}(1+|\bar{\xi}|)^{-2}d\bar{\xi}
		+C\Lambda_i^{-1}\ei^2
		\int_{\Dei}|\bar{\xi}-y|^{-1}(1+|\bar{\xi}|)^{-1}d\bar{\xi}\notag
		\\
		&+C\Lambda_i^{-1}\ei\int_{\Sei}|\xi-y|^{-2}d\sigma(\xi)\,,\notag
	\end{align}
	for $|y|\leq \delta\ei^{-1}/2$. Here, we have used the fact that $|G_i(x,y)|\leq C\,|x-y|^{-1}$ for $|y|\leq \delta\ei^{-1}/2$ and, since $v_i(y)\leq C\widetilde U_{\kappa}(y)$, $|w_i(y)|\leq C\Lambda_i^{-1}\ei$ for $|y|=\delta\ei^{-1}$. Hence, 
	$$
	|w(y)|\leq C\Lambda_i^{-1}\ei^2(\delta\ei^{-1})+C(1+|y|)^{-1}+C\Lambda_i^{-1}\ei^2\log(\delta\ei^{-1})
	+C\Lambda_i^{-1}\ei.
	$$
	This gives
	\begin{equation}\label{estim:wi}
		|w_i(y)|\leq C\,\left((1+|y|)^{-1}+\Lambda_i^{-1}\ei\right)
	\end{equation}
	for $|y|\leq \delta\ei^{-1}/2$.
	The Claim now follows from the hypothesis (\ref{hipLambda}).
	
	Now, we can use the claim above, Lemma \ref{classifLinear} and the scaling invariance of \eqref{linear:homog} to see that 
	$$w(y)=c_1\frac{\d \widetilde U_{\kappa}}{\d y_1}(y)+c_2\frac{\d \widetilde U_{\kappa}}{d y_2}(y)
	+c_3\left(\frac{1}{2}\widetilde U_{\kappa}(y)+\sum_{b=1}^{3}y^b\frac{\d \widetilde U_{\kappa}}{\d y_b}(y)\right)\,,$$
	for some constants $c_1,c_2,c_3$.
	It follows from the identities (\ref{hip:phitilde}) that 
	$$w_i(0)=\frac{\partial w_i}{\partial y_1}(0)=\frac{\partial w_i}{\partial y_2}(0)=0.$$ Thus we conclude that $c_1=c_2=c_3=0$. Hence, $w\equiv 0$. Since $|w_i(y_i)|=1$, we have $|y_i|\to\infty$. This, together with the hypothesis (\ref{hipLambda}), contradicts the estimate (\ref{estim:wi}), since $|y_i|\leq \delta\ei^{-1}/2$, and concludes the proof of Proposition \ref{estim:blowup:compl}.
	\ep


	\section{The Pohozaev sign restriction in dimension three}\label{sec:sign:restr}
	
	In this section, we prove a sign restriction for an integral term appearing in a Pohozaev type identity and some consequences for the blow-up set. 

Using Fermi coordinates $\psi:B_\delta^+\to M$, we work with the metric $\psi^*g$ on $B_{\delta}^+$ with $\partial B_{\delta}^+=S_{\delta}^+\cup D_{\delta}$ (see Section \ref{sec:pre} for notation). For simplicity, we will omit the symbol $\psi$.
For $z=(z_1,...,z_n)\in \R^n$ we set $r=|z|=\sqrt{z_1^2+...+z_n^2}$. For any smooth positive function $u$ on $B^+_{\delta}$ and $0<\rho<\delta$, we define
\begin{align*}
	P(u,\rho)&=\int_{S_{\rho}^+}\left(\frac{n-2}{2}u\frac{\partial u}{\partial r}-\frac{r}{2}|du|^2+r\left|\frac{\partial u}{\partial r}\right|^2\right)d\sigma
	+\frac{(n-2)\rho}{2n}\int_{S_{\rho}^+}Ku^{\frac{2n}{n-2}}d\sigma \\&+\frac{(n-2)\rho}{2(n-1)}\int_{\partial D_{\rho}}cu^{\frac{2(n-1)}{n-2}}d\bar{\sigma}
\end{align*}
and
$$P'(u,\rho)=\int_{S_{\rho}^+}\left(\frac{n-2}{2}u\frac{\partial u}{\partial r}-\frac{r}{2}|du|^2+r\left|\frac{\partial u}{\partial r}\right|^2\right)d\sigma\,.$$
Here, $d\sigma$ and $d\bar\sigma$ are the area elements of $S^+_\rho$ and $\partial D_\rho$ respectively, and $K, c$ are constants.
The Pohozaev identity for our situation is the following: 

\begin{proposition}\label{Pohozaev}
	If $u>0$ is a solution of
	\begin{equation*}
		\begin{cases}
			L_g u+Ku^{\frac{n+2}{n-2}}=0, \,&\text{in}\:B^+_{\rho}\,,\\
			B_gu+cu^{\frac{n}{n-2}}=0, \,&\text{on}\:D_{\rho} \,,
		\end{cases}
	\end{equation*}
then
	\begin{align}\label{eq:Pohozaev}
		P(u,\rho)&=-\int_{B_{\rho}^+}\left(z^a\partial_a u+\frac{n-2}{2}u\right)(L_g-\Delta)(u)dz
		\\		
		&\hspace{0.2cm}-\int_{D_{\rho}}\left(z^j\partial_j u+\frac{n-2}{2}u\right)(B_g-\partial_n)(u)d\bar{z},\notag
	\end{align} 
	where $\Delta$ is the Euclidean Laplacian.
\end{proposition}
\begin{proof}
	The proof is similar to \cite[Proposition 3.1]{almaraz3} using integration by parts. 
\end{proof}

Now we are ready to obtain the Pohozaev sign condition to be used in the proof of Theorem~\ref{compactness:thm}.
In what follows, assume that  $n=3$. 

\begin{theorem}\label{cond:sinal}
Let $x_i\to x_0\in \partial M$ be an isolated simple blow-up point for the sequence $\{u_i\}$ of positive solutions to the equations \eqref{eq:blow-up} with $h_{g_i}=0$.
	Suppose that $u_i(x_i)u_i\to G$ away from $x_0$, for some function $G$. Then
	\begin{equation}\label{eq:cond:sinal}
		\liminf_{r\to 0}P'(G,r)\geq 0,
	\end{equation}
where we are using Fermi coordinates centred at $x_i$ in the expression for $P'$.
\end{theorem}

\bp
We use Fermi coordinates $\psi_i$ centred at $x_i$ and omit the symbol $\psi_i$ to simplify the notation. We define $v_i(y)=\ei^{1/2}u_i(\ei y)$ for $y\in \Bei$ and have 
\begin{align}\notag
	\begin{cases}
		L_{\hat{g}_i}v_i-3\kappa v_i^{5}=0,&\text{in}\:\Bei,
		\\
		B_{\hat{g}_i}v_i+v_i^{3}=0,&\text{on}\:\Dei,
	\end{cases}
\end{align}
where $\hat{g}_i$ is the metric with coefficients $(\hat{g}_i)_{jl}(y)=(g_i)_{jl}(\ei y)$.
Observe that, from Propositions \ref{form:bolha} and \ref{estim:simples}, we know that $v_i(y)\leq C(1+|y|)^{-1}$ in $B^+_{\delta\ei^{-1}}$.

If $\Delta$ is the Euclidean Laplacian, we set
$$
F_i(u,r)=-\int_{B_r^+}(z^b\partial_bu+\frac{1}{2}u)(L_{g_i}-\Delta)u\,dz
$$
and write the Pohozaev identity \eqref{eq:Pohozaev} as
\begin{equation}\label{Pohoz}
	P(u_i,r)=F_i(u_i,r)\,.
\end{equation}
Here, the integral on $D_r$ vanishes as we are using Fermi coordinates and assuming that $h_{g_i}=0$. 
Observe that 
\begin{align*}
	F_i(u_i,r)=-\int_{B_{r\epsilon_i^{-1}}^+}(y^b\partial_bv_i+\frac{1}{2}v_i)(L_{\hat{g}_i}-\Delta)v_idy.
\end{align*}

Using the notation of Proposition \ref{estim:blowup:compl}, set $\check{U}_i(z)=\ei^{-1/2}\widetilde U_{\kappa}(\ei^{-1}z)$ and $\check{\phi}_i(z)=\ei^{-1/2}\widetilde \phi_i(\ei^{-1}z)$. We have
\begin{align}
	F_i(\check{U}_i+\check{\phi}_i,r)
	&=-\int_{B_r^+}(z^b\partial_b(\check{U}_i+\check{\phi}_i)
	+\frac{1}{2}(\check{U}_i+\check{\phi}_i))(L_{g_i}-\Delta)(\check{U}_i+\check{\phi}_i) dz\notag
	\\
	&={-}\int_{B_{r\epsilon_i^{-1}}^+}(y^b\partial_b(\widetilde U_{\kappa}+\widetilde \phi_i)
	+\frac{1}{2}(\widetilde U_{\kappa}+\widetilde \phi_i))\notag
	(L_{\hat{g}_i}-\Delta)(\widetilde U_{\kappa}+\widetilde \phi_i)dy\,.
	\notag
\end{align}
It follows from that proposition that 
\begin{equation}\label{approx:F}
	|F_i(u_i,r)-F_i(\check{U}_i+\check{\phi}_i,r)|\leq C \ei^2\int_{B_{r\epsilon_i^{-1}}^+}(1+|y|)^{-2}dy
	\leq C\ei r\,.
\end{equation}
Using \eqref{exp:g}, due to symmetry arguments and using again $h_{g_i}=0$, we have
$$
F_i(\check{U}_i+\check{\phi}_i,r)=O(\ei r)\,.
$$
Hence, $P(u_i,r)\geq -C\ei r$, which implies that
$$
P'(G,r)=\lim_{i\to \infty}\ei^{-1}P(u_i,r)\geq -Cr\,.
$$
The conclusion follows if we let $r\rightarrow 0$.
\ep

\begin{remark}
In the applications of Theorem \ref{cond:sinal} in the rest of this work, $g_i$ will be either $g$ or a rescaling of that metric. In both cases, the hypothesis $h_{g_i}=0$ is fulfilled.
\end{remark}

Once we have proved Theorem \ref{cond:sinal}, the next two propositions are similar to \cite[Lemma 8.2, Proposition 8.3]{khuri-marques-schoen} or \cite[Propositions 4.1 and 5.2]{li-zhu2}.
\begin{proposition}\label{isolado:impl:simples}
	Let $x_i\to x_0$ be an isolated  blow-up point for the sequence $\{u_i\}$ of positive solutions to the equations \eqref{eq:blow-up} with $h_{g_i}=0$. Then $x_i\to x_0$ is an isolated simple blow-up point for $\{u_i\}$.
\end{proposition}
\begin{proposition} \label{dist:unif}
	Let $\kappa, \b, R, u, C_0(\kappa,\b,R)$ and $\{x_1,...,x_N\}\subset M$ be as in Proposition \ref{conj:isolados}. If $\b$ is sufficiently small and $R$ is  sufficiently large, then there exists a constant $\overline C(\kappa, \b,R)>0$ such that if  $\max_{\partial M}u\geq C_0$ then 
	$$d_g(x_j,x_l)\geq \overline C\:\:\:\:\:\text{for all}\:1\leq j\neq l\leq N.$$
\end{proposition}
\begin{corollary}\label{Corol:8.4}
	Suppose the sequence $\{u_i\}$ of positive solutions to the equations \eqref{eq:blow-up}, with $h_{g_i}=0$, satisfies $\max_{\partial M} u_i\to\infty$. Then the set of blow-up points for $\{u_i\}$ is finite and consists only of isolated simple blow-up points.
\end{corollary}

\begin{remark}\label{rmk:corol}
The result of Corollary \ref{Corol:8.4} also holds for the case when $M$ is locally conformally flat with umbilical boundary, $n\geq 3$. In this case, a conformal change is necessary near each blow-up point in such a way that the metric becomes the Euclidean one in local coordinates. The Pohozaev sign restriction of Theorem \ref{cond:sinal} then holds for these conformal metrics around each blow-up point as the right-hand side of \eqref{eq:Pohozaev} vanishes for the Euclidean metric. As a consequence,  Propositions \ref{isolado:impl:simples} and \ref{dist:unif} are also proved. This argument is carried out in \cite{han-li1} where the reader may find more details.
\end{remark}


\section{The proof of Theorem \ref{compactness:thm}}\label{sec:pf:thm}

Before proving our main theorem, we shall recall a geometric invariant term called the mass of a manifold with non-compact boundary and recall the relevant positive mass theorem for this setting (see \cite{almaraz-barbosa-lima}).

\begin{definition}\label{def:asym}
	Let $(\hat M, \hat g)$ be a Riemannian manifold with a non-compact  boundary $\d \hat M$. Assume that $R_{\hat g}$ is integrable on $\hat M$, and $\cmedia_{\hat g}$  is integrable on $\d \hat M$.
	We say that $\hat M$ is {\it{asymptotically flat}} with order $q>(n-2)/2$, if there is a compact set $K\subset \hat M$ and a diffeomorphism $f:\hat M\backslash K\to \Rn\backslash \overline{B^+_1}$ such that, in the coordinate chart $(y_1,...,y_n)$ defined by $f$ (which we call the {\it  asymptotic coordinates} of $\hat M$), we have
	$$
	|\hat g_{ab}(y)-\delta_{ab}|+|y||\hat g_{ab,c}(y)|+|y|^2|\hat g_{ab,cd}(y)|=O(|y|^{-q})\,,
	\:\:\:\:\text{as}\:\:|y|\to\infty\,,
	$$
	where $a,b,c,d=1,...,n$, and commas are denoting derivatives.
\end{definition}

Suppose the manifold $\hat M$, of dimension $n\geq 3$,  is asymptotically flat with order $q>(n-2)/2$, as defined  above. 
Then the limit
\begin{align}\label{def:mass}
	m(\hat g)=
	\lim_{R\to\infty}\left\{
	\sum_{a,b=1}^{n}\int_{y\in\Rn,\, |y|=R}(\hat g_{ab,b}-\hat g_{bb,a})\frac{y_a}{|y|}\,\ds
	+\sum_{j=1}^{n-1}\int_{y\in\d\Rn,\, |y|=R}\hat g_{nj}\frac{y_j}{|y|}\,d\bar\sigma\right\}
\end{align}
exists, and we call it the {\it mass} of $(\hat M, \hat g)$. As proved in \cite{almaraz-barbosa-lima}, $m(\hat g)$ is a geometric invariant in the sense that it does not depend on the asymptotic coordinates.

The expression in \eqref{def:mass} is the analogue of the ADM mass for the manifolds of Definition \ref{def:asym}. A positive mass theorem for $m(\hat g)$, similar to the classical ones in \cite{schoen-yau, witten}, is stated as follows:
\begin{theorem}[\cite{almaraz-barbosa-lima}]\label{pmt}
	Assume $3\leq n\leq 7$.
	If $R_{\hat g}$, $\cmedia_{\hat g}\geq 0$, then we have $m(\hat g)\geq 0$ and the equality holds if and only if $\hat M$ is isometric to $\R_+^n$.
\end{theorem}

	The asymptotically flat manifolds we work with in this paper come from the stereographic projection of compact manifolds with boundary. Inspired by Schoen's approach \cite{schoen1} to the classical Yamabe problem, this  projection is defined by means of a Green's function with singularity at a boundary point. Since in general we do not have the same control of the Green's function expression we do in the case of manifolds without boundary, the relation with \eqref{def:mass} is obtained by means of an integral defined in \cite{brendle-chen}. 
This is the content of the next proposition which is stated for $n=3$. In the case of locally conformally flat manifolds with umbilical boundary, this type of argument is unnecessary since the Green's function has a nicely controlled expansion at infinity as we will see below in the last lines of the proof of Theorem \ref{compactness:thm}.
	\begin{proposition}\label{propo:I:mass}
		Let $(M,g)$ be a compact three-manifold with boundary $\partial M$ and consider Fermi coordinates $z$ centred at $x_0\in \d M$.  Let $G$ be a smooth positive function on $M\backslash\{x_0\}$ written near $x_0$ as
		$$
		G(z)=|z|^{-1}+\phi(z)
		$$ 
		where $\phi$ is smooth on $M\backslash\{x_0\}$ satisfying $\phi(z)=O(|\log|z|\,|)$  and set 
		\begin{align*}
			I(x_0,\rho)
			&=8\int_{S^+_{\rho}}\left(|z|^{-1}\d_aG(z)-\d_a|z|^{-1}G(z)\right)\frac{z_a}{|z|}d\sigma
			\\
			&-12\int_{S^+_{\rho}}\sum_{j,l=1}^{2}|z|^{-5}z_3z_j z_l\pi_{jl}(x_0)d\sigma\,.
		\end{align*}
		Consider the metric $\hat g=G^{4}g$ and a suppose that $R_{\hat g}$ and $h_{\hat g}$ are integrable on $M\backslash\{x_0\}$ and $\partial M\backslash\{x_0\}$ respectively, with respect to $\hat g$. Then $(M\backslash \{x_0\},\hat g)$ is asymptotically flat in the sense of Definition \ref{def:asym} with mass
		$$
		m(\hat g)=\lim_{\rho\to 0}I(x_0,\rho).
		$$
		Moreover, 
		\begin{align*}
			P'(G,\rho)=-\frac{1}{16}I(x_0,\rho)+O(\rho\, |\log \rho|),
		\end{align*}
		where $P'$ is defined in Section \ref{sec:sign:restr}.
	\end{proposition}
\bp
We refer to Propositions 3.5 and 3.6 in ~\cite{almaraz-queiroz-wang} or Propositions 5.2 and 5.3 in ~\cite{almaraz-wang} for the proof. 
\ep

\begin{proof}[Proof of Theorem \ref{compactness:thm}]
In view of standard elliptic estimates, we only need to prove that $u$ is bounded from above. As changing conformally the metric $g$ we are assuming $R_g>0$, it follows from the maximum principle that we only need to control $u$ on the boundary.
Assume by contradiction there exists a sequence $u_i$ of positive solutions of (\ref{main:equation:1}), with $0<\kappa<1$, such that 
\begin{equation*}
	\max_{\partial M} u_i\to\infty,\:\:\:\:\text{as}\:i\to\infty.
\end{equation*}
It follows from Corollary \ref{Corol:8.4} and Remark \ref{rmk:corol} that we can assume $u_i$ has $N$ isolated simple blow-up points on $\partial M$
\begin{align*}
	x_i^{(1)}\to x^{(1)},\:...\:,\: x_i^{(N)}\to x^{(N)}.
\end{align*}
Without loss of generality, suppose 
\begin{align*}
	u_i(x_i^{(1)})=\min\big\{u_i(x_i^{(1)}), ..., u_i(x_i^{(N)})\big\}\:\:\:\:\text{for all}\:i.
\end{align*}

Now for each $j=1,...,N$, consider the Green's function $G_{(j)}$ for the conformal Laplacian $L_{g}$ with boundary condition $B_{g}G_{(j)}=0$ and singularity at $x^{(j)}\in \partial M$. Observe that these Green's functions exist because $Q_g(M)>0$ by hypothesis.
In Fermi coordinates centred at the respective boundary singularities, those functions satisfy
\begin{align*}
	\big|G_{(j)}(z)-|z|^{2-n}\big|\leq 
	\begin{cases}
	C(1+|\log|z|\,|), &n=3,4,
	\\
	C|z|^{d+3-n}, &n\geq 5,
	\end{cases}
\end{align*}
according to \cite[Proposition B.2]{almaraz-sun}.

It follows from the upper bound (a) of Proposition \ref{estim:simples} that there exists some function $G$ such that $u_i(x_i^{(1)})u_i\to G$ in $C^2_{\text{loc}}(M\backslash  \{x^{(1)},...,x^{(N)}\})$. Moreover, the lower control (b) of that proposition and elliptic theory yields the existence of $a_j>0$, $j=1,...,N$, and $b\in C^2(M)$ such that 
\begin{equation*}
	G=\sum_{j=1}^{N}a_j G_{(j)}+b,
\end{equation*}
and 
\begin{align*}
	\begin{cases}
		L_{g}b=0,&\text{in}\:M,
		\\
		B_{g}b=0,&\text{on}\:\d M.
	\end{cases}
\end{align*}
The hypothesis $Q_g(M)>0$ ensures that $b\equiv 0$. 

Suppose first that $n=3$. 
If $\hat g=G_{(1)}^{4}g$, we have
$$
R_{\hat g}=-8G_{(1)}^{5}L_gG_{(1)}=0
\qquad\text{and}\qquad
h_{\hat g}=-2G_{(1)}^{3}B_gG_{(1)}=0.
$$
By Proposition \ref{propo:I:mass}, $(M\backslash\{x^{(1)}\}, \hat g)$ is an asymptotically flat manifold (in the sense of Definition \ref{def:asym}) and the mass is given by
$$
m(\hat g)=\lim_{\rho\to 0} I(x^{(1)}, \rho).
$$
Then Theorem \ref{pmt} and the assumption that $M$ is not conformally equivalent to the hemisphere gives $m(\hat g)>0$. So, by  the last statement of Proposition \ref{propo:I:mass},
$$
\lim_{\rho\to 0}P'(G_{(1)}, \rho)<0.
$$
This contradicts the local sign restriction of Theorem \ref{cond:sinal} and proves Theorem \ref{compactness:thm} in the three dimensional case.

Now suppose $M$ locally conformally flat with $\partial M$ umbilic.
We use the argument in \cite{han-li1} which we will sketch here. 
After a conformal change, we can assume that $g$ is Euclidean in local coordinates $\psi:B_\delta^+\to M$ centred at $x^{(1)}$.
In these coordinates, the Green's function $G_{(1)}$ has asymptotic expansion 
$$
G_{(1)}(z)=|z|^{2-n}+A+\alpha(z),
$$
where $\Delta\alpha=0$ in $B_\delta^+$ and $\partial_n\alpha=0$ on $D_\delta$. It follows from the positive mass theorem for locally conformally flat manifolds in \cite{schoen-yau2}, using a doubling argument and the fact that $M$ is not conformally diffeomorphic to the unit hemisphere, that $A>0$.
Then it is direct to check that 
$$
\lim_{\rho\to 0}P'(G_{(1)}, \rho)<0.
$$
As the result in Theorem \ref{cond:sinal} also holds for the Euclidean metric regardless of the dimension, this gives a contradiction and ends the proof of Theorem \ref{compactness:thm}.
\end{proof}

\bigskip\noindent
{\bf{Acknowledgement.}} The authors would like to thank the anonymous referee for his/her valuable comments and suggestions.

\bigskip\noindent
\textsc{S\'ergio Almaraz\\
Instituto de Matem\'atica e Estat\' istica, \\
Universidade Federal Fluminense\\
Rua Prof. Marcos Waldemar de Freitas S/N,
Niter\'oi, RJ,  24.210-201, Brazil}\\
e-mail: {\bf{sergioalmaraz@id.uff.br}}

\bigskip\noindent
\textsc{Shaodong Wang\\
School of Mathematics and Statistics,\\
Nanjing University of Science and Technology\\
Nanjing, 210094, People’s Republic of China} \\
e-mail: {\bf{shaodong.wang@mail.mcgill.ca}}

\end{document}